\documentclass{article}
\usepackage[left=3cm, right=3cm, top=2cm]{geometry}
\usepackage[utf8]{inputenc}
\usepackage{mathtools}
\usepackage{amsmath}
\usepackage{amssymb}
\usepackage{amsthm}
\usepackage{bbm}
\usepackage{enumerate}  
\usepackage{tikz}
\usepackage{hyperref, cleveref}
\usepackage[dvipsnames]{xcolor}
\usepackage{lineno}
\usepackage{latexsym}

\newcommand{\By}[2]{\overset{\mbox{\tiny{#1}}}{#2}}
\newcommand{\ByRef}[2]{   \By{\eqref{#1}}{#2} }

\newcommand{\leByRef}[1]{ \ByRef{#1}{\le} }
\newcommand{\geByRef}[1]{ \ByRef{#1}{\ge} }

\def\bo{\boldsymbol{\omega}}  
\def\PP{{\mathbb P}}
  
\def\EE{{\mathbb E}}
\def\se{\subseteq}
 
\def\cH{{\mathcal{H}}}  
\def\cB{{\mathcal{B}}}

\def\cE{{\mathcal{E}}}

\DeclareMathOperator{\im}{\mathrm{Im}}

\renewcommand{\epsilon}{\varepsilon}

\newtheorem{theorem}{Theorem}[section]
\newtheorem{lemma}[theorem]{Lemma}

\newtheorem{proposition}[theorem]{Proposition}

\newtheorem{definition}[theorem]{Definition}
\newtheorem{observation}[theorem]{Observation}

\title{Bounds for Hypergraph Universality}
\author{Peter Allen \and Julia Böttcher \and Jasmin Katz }
\date{\today}

\begin{document}
\maketitle
\begin{abstract}

A graph $\Gamma$ is said to be \emph{universal} for a class of graphs $\cH$ if $\Gamma$ contains a copy of every $H \in \cH$ as a subgraph. The number of edges required for a host graph $\Gamma$ to be universal for the class of $D$-degenerate graphs on $n$ vertices has been shown to be $O(n^{2-1/D}(\log n)^{2/D}(\log\log n)^{5})$. We generalise this result to $r$-uniform hypergraphs, showing the following. Given $D, r \ge 2$ and $n$ sufficiently large, there exists a constant $C = C(D, r)$ such that there exists a graph with at most 
    $$Cn^{r-1/D}(\log n)^{2/D}(\log\log n)^{2r+1}$$ 
edges which is universal for the class of $D$-degenerate $r$-uniform hypergraphs on $n$ vertices. This is tight up to the polylogarithmic term.
\end{abstract}

\section{Introduction}

The quest for universal graphs and hypergraphs is a central topic in extremal combinatorics. A hypergraph $\Gamma$ is \emph{universal} for a class of hypergraphs $\cH$ if $\Gamma$ contains a copy of every $H \in \cH$ as a (not necessarily induced) subgraph. In this case we also say that $\Gamma$ is \emph{$\cH$-universal}. Given a class of hypergraphs $\cH$, it is natural to ask what is the minimum number of edges required for a hypergraph to be $\cH$-universal.

More is known about universality for graphs than hypergraphs. Universality has been well-studied for many classes of graphs, including trees \cite{Chung_Graham_1978, Chung_Graham_1983, Note_on_trees, kim2025} and planar graphs \cite{Babai, spanning_planar, Planar_graphs}. For the class of $n$-vertex graphs with maximum degree $\Delta$, denoted $\cH_{\Delta}(n)$, an explicit construction of a $\cH_{\Delta}(n)$-universal graph with $O(n^{2-2/\Delta})$ edges was given by Alon and Capalbo \cite{AlCapExact}. It can be shown via a counting argument that $\Omega(n^{2-2/\Delta})$ edges are required to be $\cH_{\Delta}(n)$-universal, and so this result is tight. Another widely studied class of graphs are degenerate graphs. A graph is \emph{$D$-degenerate} if there exists an ordering of the vertices $v_1, \dots, v_n$ such that every vertex $v_i$ has at most $D$ neighbours in the set $\{v_1, \dots v_{i}\}$. We denote the class of $D$-degenerate graphs by $\cH^{}(n, D)$. The existence of a $\cH^{}(n, D)$-universal graph with $O(n^{2-1/D}(\log^{2/D}n)(\log\log n)^5)$ was shown by Allen, Böttcher and Liebenau \cite{pja} for all $D > 1$, and this is tight up to the logarithmic term. In the case $D = 1$ (which corresponds to trees), Chung and Graham \cite{Chung_Graham_1983} gave a construction of a $\cH^{}(n, 1)$-universal graph on $O(n \log n)$ edges, which is tight up to the multiplicative constant~\cite{Chung_Graham_1978}. There have also been recent improvements in the multiplicative constant in the lower and upper bounds for this class of graphs~\cite{Note_on_trees, kim2025}. Alon,
Dodson, Jackson, McCarty, Nenadov, and Southern~\cite{alon2024universalitygraphsboundeddensity}
derived related results concerning the class of
graphs of bounded density, which is more general than the class of graph with bounded degeneracy.

In comparison with the graph case, relatively little is known about universality for hypergraphs. However, this problem has been recieving more attention in recent years. For instance, let $\mathcal{H}^{(r)}_{\Delta}(n)$ denote the class of $n$ vertex $r$-graphs with maximum degree $\Delta$, where an \emph{$r$-graph} is an $r$-uniform hypergraph. A counting argument shows that the number of edges required for an $r$-graph to be $\mathcal{H}^{(r)}_{\Delta}(n)$-universal is $\Omega(n^{r-r/\Delta})$. Hetterich, Parczyk and Person \cite{HetterichParczykPerson} gave constructions of $\mathcal{H}^{(r)}_{\Delta}(n)$-universal $r$-graphs with $O(n^{r-r/\Delta})$ edges in the case that $r$ is even, in the case that both $r$ is odd and $\Delta = 2$, and in the case $r | \Delta$; in all other cases they also provided constructions, albeit with a polynomial error term. Nenadov \cite{nenadov2024} recently improved on this, giving a construction of a graph with $\Theta(n^{r-r/\Delta}(\log^{r/\Delta}n))$ edges which is $\mathcal{H}^{(r)}_{\Delta}(n)$-universal for all $r$ and $\Delta$, thus reducing the error term to poly-logarithmic. 

In this paper we focus on universality for $D$-degenerate $r$-graphs. In analogy to the graph case, we say an $r$-graph $G$ is \emph{$D$-degenerate} if there exists an ordering of the vertices $v_1, \dots, v_n$ of ~$G$ such that every vertex $v_i$ has at most $D$ edges contained entirely in the set $\{v_1, \dots v_{i}\}$. Equivalently, every induced $r$-subgraph has a vertex of degree at most $D$. We denote the set of $D$-degenerate $r$-graphs on $n$ vertices with maximum degree $\Delta$ by $\mathcal{H}_{\Delta}^{(r)}(n, D)$, and remove the subscript if the maximum degree is unbounded. We also note that there exists other notions of degeneracy for hypergraphs, such as skeletal degeneracy, which has been shown to be useful for Turán- and Ramsey-type problems~\cite{fox2023ramseyturannumberssparse}. The \emph{skeletal degeneracy} of a hypergraph~$G$ is simply the degeneracy of the graph obtained from~$G$ by replacing each edge by a clique.
Observe that if an $r$-graph has skeletal degeneracy $D$, then it has degeneracy at most $D/(r-1)$. Hence, the degeneracy definition we work with here is more general.

Noting that a $D$-degenerate graph may have density arbitrarily close to $D$, Nenadov~\cite{nenadov2024} showed that there exists an $\mathcal{H}_{\Delta}^{(r)}(n, D)$-universal $r$-graph with at most $Cn^{r-1/D}\log^{1/D}n$ edges. We give a result for $D$-degenerate $r$-graphs of unbounded degree, generalising the result of Allen, Böttcher and Liebenau~\cite{pja}, which gives the same bound for $r=2$.

\begin{theorem}\label{thm:main}
    Given $D, r \ge 2$ and $n$ sufficiently large, there exists a constant $C = C(D, r)$ such that there exists a graph with at most 
    $$Cn^{r-1/D}(\log n)^{2/D}(\log\log n)^{2r+1}$$ edges which is $\cH^{(r)}(n, d)$-universal. 
\end{theorem}

Theorem \ref{thm:main} is tight up to the polylogarithmic factor, as shown by the following result. In fact, we prove the following slightly stronger statement: if $\Gamma$ contains a copy of every connected $n$-vertex $D$-degenerate graph with maximum degree bounded by $rD+1$, then $e(\Gamma)\ge\tfrac{1}{100r^2D}n^{r-1/D}$. 

\begin{theorem}\label{thm:lower_bd}
    Suppose $r \ge 2$, $D \ge 1$ and let $n$ be sufficiently large. Suppose $\Gamma$ is a $\cH^{(r)}_{rD+1}(n, D)$-universal $r$-graph. Then $e(\Gamma) \ge \frac{1}{100r^{2}D}n^{r-1/D}$.
\end{theorem}

\begin{subsection}{Related Results}
Instead of solely asking for the minimum number of edges required to be universal for some class of hypergraphs, if $\cH(n)$ represents a class of hypergraphs on $n$ vertices, we can further ask what is the minimum number of edges required for an $\cH(n)$-universal hypergraph on $n$ vertices -- such a graph is called \emph{spanning-universal}. In general, this question is more difficult to answer, and other than the results for trees~\cite{Chung_Graham_1983, Note_on_trees, kim2025} and planar graphs \cite{ spanning_planar}, none of the results stated above are for spanning-universal $r$-graphs. In fact, the best upper bound for the number of edges in a $\cH^{(r)}_{\Delta}(n)$ spanning-universal $r$-graph is $O(r^{r-2/\Delta}\ln^{4/\Delta}n)$ for $r = 2$~\cite{AlonCapalbo2007} and $O(n^{r-r/(2 \Delta)} (\log^{r/(2\Delta)}n))$ for $r \ge 3$~\cite{PersonParczyk2016}. Analogous results for degenrate graphs are not known.

Another active area of research is determining the threshold probability for a random $r$-graph $H^{(r)}(n, p)$ to be universal for $\mathcal{H}^{(r)}_{\Delta}(n)$. The \emph{random graph $H^{(r)}(n, p)$} is the probability space of all labeled $r$-graphs on $n$ vertices, where each set of $r$ vertices is chosen as an edge independently with probability $p$. For $r=2$, Alon and Capalbo \cite{AlonCapalboRandom} showed that for every $\epsilon>0$ the random graph $H^{(2)}(n, p)$ is asymptotically almost surely (a.a.s.) $\mathcal{H}^{(2)}_{\Delta}((1-\epsilon )n)$-universal for $p  = O((\log n/n)^{1/\Delta})$. This result was generalised to hypergraphs by Parczyk and Person \cite{PersonParczyk2016}, who also strengthened it to a spanning-universal result when they showed that for $p = O((\log n/n)^{1/\Delta})$ the random $r$-graph $H^{(r)}(n, p)$ is $\mathcal{H}^{(r)}_{\Delta}(n)$-universal. This is the probability at which we expect every set of $\Delta$ vertices to have many common neighbours, a useful property in proving universality. However, this bound is not tight. In the $r=2$ case, Ferber and Nenadov \cite{FerberNenadov2018} have shown that $H^{(2)}(n, p)$ is $\mathcal{H}^{(2)}_{\Delta}(n)$-universal for $p = O(n^{-1/(\Delta - 0.5)} \log^3n)$, which is the current best known upper bound. The current lower bounds in the $r=2$ case are $p = \Omega(n^{2/(\Delta + 1)})$ for non-spanning universality~\cite{Conlon2017}, and $p =\Omega(n^{2/(\Delta + 1)}\log^{2/(\Delta({\Delta+1}))}n)$ for spanning universality~\cite{JKV}. 

When $D \ll \Delta$, it is possible to find tighter bounds on the values of $p$ for which the random $r$-graph is universal for the class $\mathcal{H}_{\Delta}^{(r)}(n, D)$ than for $\mathcal{H}_{\Delta}^{(r)}(n)$ alone. This has mostly been studied in the $r=2$ case, for which Ferber and Nenadov \cite{FerberNenadov2018} showed that for $p \ge (\log^3n/n)^{1/2D}$, the graph $H^{(2)}(n, p)$ is $\mathcal{H}_{\Delta}^{(2)}(n, D)$-universal. Note that the maximum degree condition is necessary -- it is not sensible to study the threshold probability for $\mathcal{H}^{(r)}(n, D)$-universality, as when we remove the requirement for bounded degree, the class $\mathcal{H}^{(r)}(n, D)$ may contain graphs with an $(r-1)$-set in $n-r+1$ edges. For any reasonable $p$, the random $r$-graph does not contain any such set of vertices.

\end{subsection}

\section{Preliminaries}
All logarithms are base 2, unless the base is specified. For a set $S$, we let $\binom{S}{k}$ denote the set of $k$-element subsets of $S$. In this notation, an $r$-graph $H = (V, E)$ is a hypergraph where each edge $s \in E$ is an element of $\binom{V}{r}$. 
Let $H = (V, E)$ be an $r$-graph. Given an $r$-graph $H$ and subsets $X_1,\dots,X_r$ of its vertices, we write $E(X_1,\dots,X_r)$ for the set of edges of $H$ with one vertex in each of $X_1,\dots,X_r$.

For some $v \in V(H)$, we denote the \emph{link} of $v$ by $L(v) = \{X \in \binom{V}{r-1} : X \cup v \in E(H)\}$. Suppose we have an ordering of vertices $v_1, \dots, v_n$ of an $r$-graph $H$. Denote by $L^-(v_i) = L(v_i)\cap {\binom{\{v_1,  \dots, v_{i-1}\} }{r-1}}$ the \emph{back-link} of~$v_i$, that is, the subset of the link of $v_i$ consisting of sets containing only vertices that precede $v_i$ in our ordering of $V(H)$. 

By definition, for any $D$-degenerate $r$-graph there exists a \emph{$D$-degenerate ordering} of the vertices $v_1, \dots, v_n$ such that for each $i \in [n]$ we have $|L^-(v_i) |  \le D$. Hence, for any $D$-degenerate $r$-graph $H$ on $n$ vertices, it is clear that $H$ has less than $Dn$ edges. This implies the following observation. 

\begin{observation}\label{obs:vtxdeg}
    Let $H$ be a $D$-degenerate $r$-graph, and let $X$ be the set of vertices in $H$ of degree at least $k$. We have $\frac{|X| k}{r} \le |E(H)| \le Dn$. Hence $|X| \le \frac{rDn}{k}$. 
\end{observation}

We will use the following basic inequality.

\begin{theorem}[Stirling's Approximation]
For integers $n,k$ and $e$ Euler's number, 
\begin{equation}\label{eq:stirling_v1}
    \lim_{n \to \infty} \frac{n!}{{\sqrt{2 \pi n}}(n/e)^n} =1,
\end{equation}
which implies
    \begin{equation}\label{eq:stirling}
    \left(\frac{n}{k} \right)^k \le \binom{n}{k} \le \left(\frac{en}{k} \right)^k.
\end{equation}
\end{theorem}
We also use the following concentration inequality.

\begin{theorem}[Chernoff bound \cite{Chernoff_one_sided}] \label{chernoff}
    Let $X = \sum_{i = 1}^nX_i$, where the $X_i$ are independent Bernoulli random variables. Then
    \begin{enumerate}[(i)]
        \item $\PP\big(X \ge (1 + \delta) \EE X\big) \le \exp\big(-\tfrac{\delta^2}{2 + \delta}\EE X\big)$ for all $\delta > 0$,\label{chernoff_upper}
        \item $\PP \big(X \le (1-\delta)\EE X\big) \le \exp\big(-\tfrac{\delta^2}{2}\EE X\big)$ for all $0 < \delta < 1$.\label{chernoff_lower}
    \end{enumerate}
\end{theorem}

\section{Lower Bound}
In this short section we give a proof of Theorem \ref{thm:lower_bd}, which follows from a counting argument. 
\begin{proof}[Proof of Theorem \ref{thm:lower_bd}]
    Let $s$ be the number of connected graphs in $\cH^{(r)}_{rD+1}(n, D)$ on vertex set $\{1,\dots,n\}$ where $1,\dots,n$ is a $D$-degeneracy ordering. We can construct such a graph as follows. For each $r \le i \le n$, we add between $1$ and $D$ edges between vertex $i$ and the preceding vertices, with the restriction that any vertex $v \in [i-1]$ can only be included in an edge with vertex $i$ if the degree of $v$ is currently at most $rD$. By Observation \ref{obs:vtxdeg}, the number of vertices of degree at least $rD +1$ in $[i-1]$ is at most $\frac{rD}{rD + 1}(i-1)$. This implies there are at least $\frac{1}{rD+1}(i-1)$ vertices of degree at most $rD$ within $[i-1]$. Each edge added at this step will use  $r-1$ of these vertices, so by~\eqref{eq:stirling} the number of choices at this step is
    \begin{equation*}
        \binom{{\frac{1}{rD+1}(i-1)}}{r-1} \ge \left(\frac{i-1}{(r-1)(rD+1)}\right)^{r-1}.
    \end{equation*}
    This is at least $\frac{1}{2r^2D}i^{r-1}$ when $n$ is sufficiently large and $i\ge \sqrt{n}$. Thus, the number of choices we can make in the last $n-\sqrt{n}$ steps is at least
    \begin{align}\label{eq:lower_bd_s}
        \prod_{i=\sqrt{n}}^n\binom{{i^{r-1}}/{(2r^2D)}} {D}
        \geByRef{eq:stirling} \left(\frac{1}{2r^2D^2}\right)^{Dn}\left(\frac{n!}{\sqrt{n}!}\right)^{(r-1)D} \geByRef{eq:stirling_v1} 10^{-Dn}(rD)^{-2Dn}n^{(r-1)Dn}.
    \end{align}
    We conclude that $s \ge 10^{-Dn}(rD)^{-2Dn}n^{(r-1)Dn}$. 

    Suppose $\Gamma$ is a $\cH^{(r)}_{rD+1}(n, D)$-universal graph. For an upper bound on $s$, we count the subgraphs of $\Gamma$ which are connected and have exactly $n$ vertices, together with an ordering of their vertices. Any connected graph in $\cH^{(r)}_{rD+1}(n, D)$ has between $n-1$ and $Dn$ edges, so for each $n-1 \le q 
    \le Dn$ we can pick $q$ edges of $\Gamma$. If these edges span exactly $n$ vertices, we assign all $n!$ possible labelings to these vertices. The number of labeled graphs we construct in this way is at most  
    \begin{equation*}
        \sum_{q={1}}^{Dn} \binom{e(\Gamma)}{{q}}n! \le Dn\binom{e(\Gamma)} {{Dn}}n! \leByRef{eq:stirling} 10^{Dn}(Dn)^{-Dn}e(\Gamma)^{Dn}n^n,
    \end{equation*}
    where we use $e(\Gamma) \ge 2Dn$. As $\Gamma$ is $\cH^{(r)}_{rD+1}(n, D)$-universal, this is an upper bound on $s$. Using~\eqref{eq:lower_bd_s} gives
    \begin{equation*}
        10^{-Dn}(rD)^{-2Dn}n^{(r-1)Dn} \le 10^{Dn}(Dn)^{-Dn}e(\Gamma)^{Dn}n^n.
    \end{equation*}
    Rearranging and taking roots gives the result.
\end{proof}

\section{Random Block Model and Proof of Theorem \ref{thm:main}}
The proof of Theorem \ref{thm:main} generalises the random block model of Allen, Böttcher and Liebenau \cite{pja}. That is, we will define an $r$-graph $\Gamma$ on $\Theta(n)$ vertices, whose vertices are split into $\Theta(\log\log n)$ \emph{blocks} with sizes growing from $n^{1-1/D}$ up to $n$, and with edges placed between and in blocks randomly with probability depending on the size of the blocks they lie between.

We will then show that any $D$-degenerate graph can be embedded in $\Gamma$, putting the vertices of highest degree in the smallest block down to constant degree in the largest block.

\begin{definition}[Random block model, $\Gamma(r,  n, D)$]\label{def:RB}
Given any natural numbers $r, D, n \ge 2$, we define $N$ to be the smallest integer such that $n^{D^{1-N}} \le 3^{D^2}$. For each $0 \le k \le N $, we define the variable
\begin{equation*}\label{eq:deltas}
    \Delta_k := 
\begin{cases}
n^{D^{-k}} & \text{for $0 < k \le N$}\\
Dn  & \text{if $k = 0$.}\\
\end{cases}
\end{equation*}
The vertex set of $\Gamma: = \Gamma(r,  n, D)$ is the disjoint union $W = W_1 \sqcup \hdots \sqcup W_N$, where each $W_k$ is called a \emph{block} and has order 
\begin{equation*}
    |W_k| 
    =100 \cdot 3^Drn/\Delta_k .
\end{equation*}

For any subset $s \se W$, we define an \emph{intersection pattern} $\pi \in \{0, \hdots, |s|\}^N$, where for each $j \in [N]$, $\pi_j (s) = |s \cap W_j|$. We define \[p^*: = 2^{(r-1)}(2(r-1)D)^{1/D}(\log  n)^{2/D}(\log\log n)^{r+1}/\Delta_1.\] For each $s \in \binom{W}{r}$, define 
\begin{equation*}\label{eq:prob}
p_{s} := 
\min\bigg\{1, p^* \prod_{i=1}^N\Delta_i^{{\pi_i(s)}} \bigg\}.
\end{equation*}
The edge set $E(\Gamma)$ is then defined by letting each set $s \in \binom{W }{r}$ be an edge of our hypergraph independently with probability $p_{s}$.

We further partition each block $W_k$ into \emph{sub-blocks} $W_{k, 1} \hdots W_{k, \log n}$, where $W_{k, 1} = \frac 1 2 |W_k|$ and for $j \ge 2$, $|W_{k, j}| \ge \frac{|W_{k}| }{2\log n} $.
\end{definition}

The meaning of $\Delta_k$ is the following: when we embed a $D$-degenerate graph to $\Gamma$, we will embed vertices of degree between $\Delta_{k}$ and $\Delta_{k-1}$ to the block $W_k$.

We now give some properties of this model.

\begin{proposition}[Properties of the Model] \label{modelprops}
    For sufficiently large $n$, $D \ge 2$ and for $N, \Delta_{N-1}$ as in Definition~\ref{def:RB}, the following holds.
    \begin{enumerate}[(a)]
        \item $\frac{\log\log n}{2\log D} \le N < \log\log n$ and $3^D \le \Delta_{N-1} \le 3^{D^2}$ \label{props1},
        \item ${100rn} \le |W_N| \le 100 \cdot 3^{D-1}rn$,\label{propsW_N}
        \item $\Gamma(r, n, D)$ has at most $200\cdot 3^Drn$ vertices and asymptotically almost surely at most \[2(20\cdot 3^Dr)^r(\log\log n)^{2r+1}(\log n)^{2/D}n^{r-1/D}\] edges. \label{props3_edges}
    \end{enumerate}
\end{proposition}

\begin{proof}
    The first upper bound in Proposition \ref{modelprops}~\eqref{props1} follows from using $n^{D^{1-(N-1)}} > 3^{D^2}$ to obtain $N < \log_D(\log_3(n)) < \log \log n$. The other inequalities can be similarly calculated from the definition of $N$. This then implies~\eqref{propsW_N} as $|W_N| = \frac{3^Drn}{\Delta_{N-1}^{1/D}}$. 

    The upper bound on the number of edges given in~\eqref{props3_edges} can be computed by calculating the expected number of edges within and between the blocks. Fix some $r$-tuple $S = (i_1, \hdots, i_r)$, where the $i_\ell$ are chosen from  $[N]$. For each $k \in [N]$, let $m_S(k)$ denote the multiplicity of $k$ in $S$. We then have 
    \begin{align*}
        \EE(|E(W_{i_1}, \dots, W_{i_r})|) &\le \min\bigg\{1, p^*\prod_{\ell = 1}^r\Delta_{i_\ell} \bigg\}\cdot  \Bigg(\prod_{\substack{k \in [N]}} |W_k|^{m_S(k)} \Bigg)\,.
        \end{align*}
    Here the inequality accounts for the cases where some of the $W_{i_j}$ are the same. Hence for each $r$-tuple $S = (i_1, \dots, i_r)$, 
        \begin{align*}
        \EE(|E(W_{i_1}, \dots, W_{i_r})|) &\le p^*\prod_{\ell = 1}^r\Delta_{i_\ell}  \cdot 100 \cdot 3^Drn/\Delta_{i_\ell}\\
         &= p^*100 \cdot 3^{rD}r^rn^{r} \\
         &\le (20\cdot 3^Dr)^r(\log\log n)^{r+1}(\log n)^{2/D}n^{r-1/D} .
    \end{align*}
  Summing over all $N^r \le(\log \log n)^r$ $r$-tuples gives \[\EE(|E(\Gamma)| \le (20\cdot 3^Dr)^r(\log\log n)^{2r+1}(\log n)^{2/D}n^{r-1/D}.\]

  We also clearly have $n^{r-r/D} < |E(W_1)| \le \EE(|E(\Gamma)| )$, hence by the Chernoff bound in Theorem \ref{chernoff}~\eqref{chernoff_upper} with $\delta= 1$, 
  \begin{equation*}
      \PP\big(|E(\Gamma)|> 2\cdot (20\cdot 3^Dr)^r(\log\log n)^{2r+1}(\log n)^{2/D}n^{r-1/D}\big) \le \exp(-(n^{r-r/D})/3) = o(1).
  \end{equation*} 
  
\end{proof}

Informally, our embedding strategy is as follows. We embed the vertices of a $D$-degenerate $H$ one by one in degeneracy order into $\Gamma$, at each time $t$ choosing to embed the next vertex $v_t$ into the block $W_k$ according to the degree of $v_t$ being between $\Delta_k$ and $\Delta_{k-1}$, and choosing the first sub-block of $W_k$ where this is possible without having embedded all vertices of an edge of $H$ to a non-edge of $\Gamma$. Formalising, this, we obtain:

\begin{definition}[Embedding Strategy]\label{def:embedding_strat}
    Given any $D$-degenerate $r$-graph $H$, let $v_1, \hdots, v_n$ be a  $D$-degeneracy ordering of $V(H)$, and let $\Gamma$ have the properties of Proposition~\ref{modelprops}.  Our embedding strategy is defined by the iterative construction of a \emph{partial embedding $\psi_t : \{v_1, \dots, v_t\}  \to \Gamma$}. We start with $\psi_0$, the trivial partial embedding of no vertices into $\Gamma$. 
    
    For each $1 \le t \le n$, we say a vertex $u \in \Gamma$ is a \emph{candidate} for $v_t$ if $\psi_{t-1}(L^-(v_t)) \se L(u)$. We say such a vertex $u\in \Gamma$ is an \emph{available candidate} if in addition $u \notin \im(\psi_{t-1})$. Let $k$ be minimal such that $\Delta_k < \deg(v_t)$. If there is no such $k$ choose $k = N$. Choose $j$ minimal such that there is some $u \in W_{k, j}$ which is an available candidate for $v_t$. We define $\psi_t=\psi_{t-1}\cup\{v_t\to u\}$. If there is no available candidate for $v_t$ in  $W_k$, we say $\psi_t$ and the subsequent partial embeddings do not exist and that the embedding strategy \emph{fails at time $t$}. If the embedding strategy does not fail, then $\psi: = \psi_n$ gives an embedding of $H$ into $\Gamma$.
\end{definition}

The requirement for each vertex to be embedded to an available candidate ensures that $\psi$ is indeed an embedding. The additional requirement for each vertex $v\in V(H)$ to be embedded to the block $W_k$ for $k$ such that $\Delta_k < \deg(v)$ further ensures (by Observation \ref{obs:vtxdeg}) that each block always contains vertices that are not embedded to.

Given any partial embedding $\psi_t$ and any vertex $v \in H$ whose back-link is in the domain of $\psi_t$ (i.e.\ is embedded), we define the \emph{embedded back-link of $v$} by $\psi_t(L^-(v)) \se \binom{W }{ {r-1}}$. To simplify notation for any subset $B \se \binom{W}{r-1}$ we let $V(B) = \bigcup B \se W$ denote the set of vertices contained in $B$. 

In our proof that this embedding strategy does not fail, we analyse the multi-set of embedded back-links of all the vertices in some sub-block, rather than just the embedded back-link of a single vertex. This may indeed be a multi-set, as two vertices may share the same back-links and for our proof we need to count with multiplicity. We will consider a multi-set $\cB$, where each $B \in \cB$ is a subset of $\binom{W} {{r-1}}$. Let $\cB$ be some such multi-set. We denote by $|\cB|$ the number of (not necessarily distinct, so counting with multiplicity) sets contained in $\cB$, and call this the \emph{size} of $\cB$. We let $V(\cB) \se W$ denote the set of vertices contained in the union of all $\bigcup_{B \in \cB}V(B)$ where $B \in \cB$. We now define three key properties a multi-set of the vertices of $\Gamma$ should exhibit if they are the embedded back-links of vertices. The first and second properties are obvious consequences of our embedding strategy. For the third, we will give a deterministic condition that $\Gamma=\Gamma(r,n,D)$ is likely to satisfy and we will show this condition implies that our embedding strategy also maintains the third property.

\begin{definition}[Well-behaved multi-set]\label{def:well}
  Given $r, n, D \ge 2$, let $W = \bigcup_{k} W_k$ and $W_k = \bigcup_{j} W_{k,j}$ be the vertex set of $\Gamma=\Gamma(r, n, D)$ as given by Definition \ref{def:RB}.  For $1\le t \le n,$ let $\mathcal{B} = \{B_i\}_{i=1}^t$ be a multi-set, where each $B_i\subseteq \binom{W }{{r-1}}$. Then $\cB$ is called \emph{well-behaved} if
  \begin{enumerate} [(WB1)]
  \item  \label{WB1}
    $|B_i| \le D$  for all $1\le i \le t$, 
  \item \label{WB2} 
    for all $1\le k\le N$ and for all $u\in W_k$ we have $\big|\{ i \in [t] : u \in V(B_i) \} \big| \le \Delta_{k-1}$,  and 
  \item \label{WB3}
    for each $1\le k\le N$ and each $1\le j\le\log n$, we have $\big|V(\mathcal{B})\cap W_{k,j}\big|\le\tfrac12|W_{k,j}|$.
  \end{enumerate} 
\end{definition} 

Our goal is to find that for each $k \in [N]$ and $j \in [\log n]$ and any well-behaved multi-set $\cB$, some significant proportion of vertices in $W_{k, j}$ have link containing a set from $\cB$. In terms of back-links, if $\cB$ is the multi-set of embedded back-links of some $x_1,\dots,x_t\in V(H)$, then we are saying that a significant proportion of vertices in $W_{k,j}$ are candidates for at least one of $x_1,\dots,x_t$. This is stated precisely in the following Lemma. 

\begin{lemma}\label{lem:pseudorandom}
Let $D, r \ge 2$ and let $n$ be sufficiently large. Then $\Gamma(n, r, D)$ as given in Definition \ref{def:RB} a.a.s.\ satisfies the following. For every $t \in [n]$, every well-behaved multi-set $\cB$ containing $t$ subsets of $ \binom{W}{{r-1}}$, and every $k \in [N]$ and $j \in [\log n]$, we have 
\begin{equation}\label{eq:lem}
    \big|\big\{u \in W_{k, j} : \exists B \in \cB \text{ such that } B \se L(u)\big\}\big|  \ge \min\Big\{\frac{1}{16}, \frac t 4 (\log  n)^{2}(\log\log n)^{D}n^{D^{1-k}-1}
    \Big\}|W_{k, j}|. 
\end{equation}
\end{lemma}

The proof of Lemma \ref{lem:pseudorandom} will be deferred to Section \ref{sec:lemmas}. We now directly show how Lemma \ref{lem:pseudorandom} implies Theorem \ref{thm:main}.

\begin{proof}[Proof of Theorem \ref{thm:main}]
Let $\Gamma(n, r, D)$ be as given in Definition~\ref{def:RB}. Then by Proposition~\ref{modelprops}~\eqref{props3_edges} and Lemma~\ref{lem:pseudorandom}, we have that $\Gamma$ a.a.s.\ has at most 
\begin{equation*}
    2(20\cdot 3^Dr)^r(\log\log n)^{2r+1}(\log n)^{2/D}n^{r-1/D}
\end{equation*} edges, and satisfies the following. For every $t \in [n]$, every well-behaved multi-set $\cB$ of $ \binom{W}{{r-1}}$, and every $k \in [N]$ and $j \in [\log n]$, we have 
\begin{equation}\label{eq:pseudorandom}
    \big|\big\{u \in W_{k, j} : \exists B \in \cB \text{ such that } B \se L(u)\big\}\big| \ge \min\Big\{\frac{1}{16}, \frac {|\cB|} 4 (\log  n)^{2}(\log\log n)^{D}n^{D^{1-k}-1}
    \Big\}|W_{k, j}|. 
\end{equation}

Let $H$ be a $D$-degenerate graph on $n$ vertices, and let  $v_1, \hdots, v_n$ be a an ordering of its vertices, such that for each $v_i$ we have $L^-(v_i) \le D$. We use the embedding strategy of Definition \ref{def:embedding_strat}, which gives a partial embedding $\psi_t$ of $\{v_1, \hdots, v_t\}$. Note that this strategy guarantees any set of embedded back-links satisfies (WB\ref{WB1}) and (WB\ref{WB2}). We want to prove this embedding strategy is successful. Suppose, for a contradiction, the embedding fails at time $t^*$ and let ${k^*}$ be such that $\Delta_{k^*} < v_{t^*} \le \Delta_{{k^*}-1}$. 

By definition, the embedding failed because there was no available candidate for $v_{t^*}$ in $W_{k^*}$. We use~\eqref{eq:pseudorandom} to give a lower bound on the number of candidates in $W_{k^*}$. Let $\{b_1, \hdots, b_m\}=L^-(v_{t^*})$, and let  
$B = \{\psi_{t^*-1}(b_1), \hdots, \psi_{t^*-1}(b_m)\}$. Clearly $\{B\}$ is well-behaved as $|V(\{B\})| \le(r-1)D < \tfrac{1}{2}|W_{k, j}|$ for all $k, j$, so by \eqref{eq:pseudorandom} the number of vertices $u$ in $W_{{k^*}, \log n}$ with $B \se L(u)$ is at least
\begin{align}
    &\min\Big\{\frac{1}{16}, \frac {|\cB|} 4 (\log  n)^{2}(\log\log n)^{D}n^{D^{1-k}-1}
    \Big\}|W_{k, j}|
    \\ &\ge \min\Big\{\frac{100 \cdot 3^Drn}{16\Delta_k}, \frac {100\cdot 3^D rn|\cB|  (\log  n)^{2}(\log\log n)^{D}\Delta_k^D}{4n\Delta_k}
    \Big\}
    \nonumber\\
    & \ge \min\Big\{5 \cdot 3^Drn^{1-D^{-k}},  25\cdot 3^D r (\log  n)^{2}(\log\log n)^{D}n^{D^{1-k}-D^{-k}}\big\} \nonumber
    \\&> 1. \label{eq:cand_set_size}
\end{align}

As the embedding strategy failed at time $t^*$, these candidates must have been covered by $\im \psi_{t^*-1}$. In order to find a contradiction, we define the following constants.
\begin{equation*}
    L_{k, j} =
\frac{1}{(4\log n)^{j-1}}\cdot
\begin{cases}
\frac{rnD}{\Delta_k} & \text{for $k \in [N-1]$ }\\
n & \text{for $k = N$.}\\
\end{cases}
\end{equation*}
We show inductively that for all $t \in [t^*-1]$, each $L_{k, j}$ is an upper bound on $|\im\psi_t \cap W_{k, j}|$ (note that $t^* >1$, as there is always an available candidate in the appropriate block to embed the first vertex). In particular, this will show that 
\begin{equation*}
    |\im\psi_{t^*-1} \cap W_{{k^*}, \log n}| \le L_{{k^*}, \log n} \le \frac{1}{(4\log n)^{\log n-1}}n < \frac{1}{5^{\log n}}n  <  1,
\end{equation*}
and hence by~\eqref{eq:cand_set_size} there is an available candidate for $v_{t^*}$ in $W_{k^*}$, a contradiction. 

We proceed with the induction. In the base case we embed the first vertex to some $W_{k, 1}$, and we have $|\im \psi_1 \cap W_{k, 1}|  = 1 \le L_{k, 1}$. For all other sub-blocks $W_{k, j}$ we trivially have $|\im \psi_1 \cap W_{k, j}| = 0 \le L_{k, j}$.

We now let $1 < t \le t^* -1$, and assume that for $i = t -1$ and each $k \in [N]$ and $ j \in [\log n]$, we have $|\im\psi_i \cap W_{k, j}| \le L_{k, j}$. Suppose, for a contradiction, that $L_{k, j}$ is not an upper bound on $|\im\psi_t \cap W_{k, j}|$. As this upper bound held at time $t-1$, it must have been violated when we embedded $v_t$. Suppose $v_t$ was embedded to $W_{k', j'}$. Then 
\begin{align}
    & L_{k', j'} < |W_{k', j'} \cap \im\psi_{t}|  \le L_{k', j'}+1   , \text{ and } \nonumber\\
    &|W_{k, j} \cap \im\psi_{t}|  \le L_{k, j}  \text{ for all $(k,j) \in [N]\times [\log n] \setminus \{(k',j')\}$}.\label{eq:limits}
\end{align}

First, suppose $j' =1$. If $k' = N$, then $|W_{N, 1} \cap \im\psi_{t}|  > L_{N, 1}  = n$, a contradiction to $|V(H)| = n$. If $k' < N$, then $|W_{k', 1} \cap \im\psi_{t}|  > L_{k', 1}  = \frac{rnD}{\Delta_{k'}}$, a contradiction to Observation \ref{obs:vtxdeg}. 

Now suppose $j' > 1$. We let $\cB =\{\psi_t(L^-(x)) : \psi_t(x) \in W_{k', j'}\}$ be the multi-set of embedded back-links of all the vertices which were embedded to $W_{k', j'}$. We will show that these back-links are well-behaved, and that~\eqref{eq:limits} and~\eqref{eq:pseudorandom} together imply that there is an available candidate for $v_t$ in $W_{k', j'-1}$, which is the contradiction we require. As mentioned above, properties (WB\ref{WB1}) and (WB\ref{WB2}) are consequences of our embedding strategy, so it only remains to verify (WB\ref{WB3}). We note that $V(\cB) \se \im\psi_{t-1}$, and prove the stronger property that for all $1 \le k \le N$ and $ 1 \le j \le \log n$, we have $|\im\psi_{t-1} \cap W_{k, j}| \le L_{k, j} + 1 \le \frac{1}{16}|W_{k, j}|$. Recall that 
\begin{align}
    L_{k, j} &\le\frac{L_{k, 1}}{4\log n} &&\text{for all $k \in [N]$ and $ j> 1$}, \nonumber\\
    |W_{k, 1}|&=50 \cdot 3^Drn/\Delta_k &&\text{for $k < N$}, \label{eq:Wk1_size}\\
    |W_{N, 1}| &\ge 50rn  &&\text{by Proposition \ref{modelprops}~\eqref{propsW_N}, and} \nonumber\\
    |W_{k, j}| &\ge \frac{|W_{k,1}|}{\log n} &&\text{for all $k \in [N]$ and $j \ge 1$}. \label{eq:Wkj_size}
\end{align}
Hence, for $j =1$, 
\begin{equation}\label{eq:j=1}
\begin{split}
L_{N, 1} + 1  &= n + 1 < \frac{50}{16}rn \le \frac{1}{16}|W_{N, 1}|,\\[6pt]
L_{k, 1} + 1  &= rnD/\Delta_k + 1  < \frac{50}{16}3^Drn/\Delta_k  = \frac{1}{16}|W_{k, 1}| \text{ for $k<N$.}
\end{split}
\end{equation}
For $j > 1$ and all $k \in [N]$, 
\begin{align*}
    &L_{k, j} + 1 \le \frac{1}{4\log n}L_{k, 1} + 1 \leByRef{eq:j=1} \frac{1}{64 \log n}|W_{k, 1}| + 1 \leByRef{eq:Wkj_size} \frac{1}{64}|W_{k, j}| + 1 < \frac{1}{16}|W_{k, j}|.
\end{align*}

 This completes the check that $\cB$ is well-behaved. We can therefore apply~\eqref{eq:pseudorandom} to $\cB$ with $k'$ as $k$ and $j'-1$ as $j$. By \eqref{eq:Wk1_size} and \eqref{eq:Wkj_size} we have $|W_{k', j'-1}| \ge \tfrac{3^Drn}{2\Delta_{k'}\log n}$. Since $|\cB| = |W_{k', j'} \cap \im\psi_{t}| > L_{k', j'}$, we get
\begin{align*}
    |\{u \in W_{k', j'-1} : \exists B \in \cB \text{ with } B \se L(u)\}| &\ge \min\Big\{\frac{1}{16}, \frac {L_{k', j'}}{4} \frac{(\log  n)^{2}(\log\log n)^{D}n^{D^{1-k'}}}{n}\Big\}|W_{k', j'-1}| \\
    & = \min\Big\{\frac{1}{16}, \frac {L_{k', j'-1}}{16\log n} \frac{(\log  n)^{2}(\log\log n)^{D}\Delta_{k'}^D}{n}\Big\}|W_{k', j'-1}| \\
    &=\min\Big\{\frac{|W_{k', j'-1}|}{16}, \frac {L_{k', j'-1}3^Dr\Delta_{k'}^{D-1}(\log  n)^{2}(\log\log n)^{D}}{32 (\log n)^2}\Big\} \\
    & \ge L_{k', j'-1}+1.
\end{align*}
As $L_{k', j'-1}$ is an upper bound on the number of vertices embedded into $W_{k', j'-1}$, by induction there is some vertex in $W_{k', j'-1}$ which is an available candidate for $v_t$, a contradiction to the definition of $j'$. This concludes the proof that for all $t \in [t^*-1]$, each $L_{k, j}$ is an upper bound on $|\im\psi_t \cap W_{k, j}|$. Hence there is an available candidate for $t^*$ in $W_{k^*}$, and our embedding strategy never fails. 
\end{proof}

\section{Proof of Lemma \ref{lem:pseudorandom}} \label{sec:lemmas}

In this section, we prove Lemma \ref{lem:pseudorandom}. We first show that if we have some well-behaved multi-set $\cB$, then the link of each vertex of the block model contains some $B \in \cB$ with reasonable probability. Lemma \ref{lem:pseudorandom} will then follow via an application of the Chernoff bound to a fixed sub-block and well-behaved multi-sets, followed by a union bound over all possible sub-blocks and well-behaved multi-sets.

\begin{lemma}\label{lem:cand_for_some}
    Let $D, r \ge 2$ and let $n$ be sufficiently large and $t \in [n]$. Let $\Gamma(n, D, r)$ be the random block model as in Definition \ref{def:RB}. Let $\cB$ be a well-behaved multi-set with $|\cB| = t$. Fix $1 \le k \le N$ and any $v \in W_k \setminus V(\cB)$. Let $\cE$ be the event that there is some $B \in \cB$ such that $B \se L(v)$. Then 
    \begin{equation*}
        \PP(\cE) \ge \min\{1/2, t(\log  n)^{2}(\log\log n)^{D}n^{D^{1-k}-1}\}.
    \end{equation*}
\end{lemma}
\begin{proof}
Throughout this proof, for each $b \se V(\Gamma)$ and $v \in V(\Gamma)$ we will slightly abuse notation and write $v \cup b$ to denote $\{v\} \cup b$.

If there is some $B \in \cB$ such that $\PP(B \se L(v) ) = 1$ then $\PP(\cE) = 1$ and we are done, so we assume $\PP(B \se L(v) ) < 1$ for all $B \in \cB$. We have that for each $B \in \cB$, the probability that $B$ is contained in the link of $v$ is given by $\PP(B \se L(v) ) =\prod_{b \in B} p_{v \cup b}$, where the empty product is equal to one. We further assume that for each $B \in \cB$, we have $p_{v \cup b} < 1$ for all $b \in B$, as removing the sets with $p_{v \cup b} = 1$ from $B$ does not change the probability $\PP(B \se L(v))$. This means that when we write $p_{s}$ where $s=v\cup b$ for some $b\in B\in\cB$, the minimum in the formula defining this probability as in equation~\eqref{eq:prob} is $p_s=p^* \prod_{i=1}^N\Delta_i^{{\pi_i(s)}}$.

We generalise the intersection pattern $\pi$ of a set $s \se [W]$ to a \emph{block intersection pattern} $\bo(B)$ for each $B \in \cB$, given by $\bo(B) \in \{0, \hdots, (r-1)D\}^N$, where $\bo_i(B)=  \sum_{b \in B}\pi_i(b)$. We now restrict our attention to the subset  $\cB' \se \cB$ of blocks with a most common intersection pattern, $\bo'$. Suppose $|\cB'| = t'$ and let $\cE'$ be the event that there is some $B \in \cB'$ such that $B \se L(v)$. We will show
     \begin{equation} \label{eq:same_ip}
        \PP(\cE') \ge \min\Big\{\tfrac 1 2, 2^{(r-1)D}{t'}(\log n)^2(\log\log n)^{rD}n^{D^{1-k}-1}\Big\}.
    \end{equation}

To show that this implies the lemma, we give an upper bound on the number of different block intersection patterns defined by the elements of $\cB$, which will give a lower bound on $t'$. Let $B \in \cB$, where $B = \{b_1, \hdots, b_k\}$ for some $k \le D$. To determine its intersection pattern, we need for each $b_i$ a vector of $r-1$ elements of $N$ that indicates which $W_j$ the elements of $b_i$ are contained in. We can arrange these vectors as the columns of an $(r-1) \times D$ matrix, where each entry is either from $[N]$, or empty (if $|B|<D$). Matrices of this form determine all possible intersection patterns, and there are therefore at most $(N+1)^{(r-1)D}\le (2\log\log n)^{(r-1)D}$ possible intersection patterns. As $t'$ is the most common, we have $t' \ge t/(2\log \log n)^{(r-1)D}$ and~\eqref{eq:same_ip} implies 
    \begin{align*}
        \PP(\cE) \ge \PP(\cE') & \ge  \min\big\{\tfrac12, 2^{(r-1)D}{t'}(\log n)^{2}(\log\log n)^{rD}n^{D^{1-k}-1}\big\}
       \\ &\ge \min\Big\{\frac12, \frac{2^{(r-1)D}t}{(2\log\log n)^{(r-1)D}}(\log  n)^{2}(\log\log n)^{rD}n^{D^{1-k}-1}\Big\}\\
        &= \min\big\{\tfrac12, t(\log n)^2 (\log\log n)^{D}n^{D^{1-k}-1}\big\}.
    \end{align*}
    
    It only remains to prove~\eqref{eq:same_ip}. For each $B \in \cB'$, let $X_B$ be the indicator random variable for the event $B \se L(v)$ and let $X = \sum_{B \in \cB'}X_B$. Then by Chebyshev's inequality
    \begin{equation}\label{eq:Chebyshev}
        \PP(\cE') = \PP(X > 0) \ge \frac{(\EE X)^2}{\EE X^2}.
    \end{equation}
To find a lower bound for the right-hand side of~\eqref{eq:Chebyshev}, we will upper bound $\EE X^2$. To that end, we write
\begin{align}\label{eq:EX2}
        \EE X^2 = \sum_{B, A \in \cB'}\PP(X_B = 1, X_A =1) = \sum_{B\in \cB'} \Big( \PP(B \se L(v))\cdot \sum_{A\in \cB'}\PP\big(A \setminus B \se L(v)\big)\Big).
\end{align}

Fix $B \in \cB'$. We have
\begin{align}
    \sum_{A\in \cB'}\PP\big(A \setminus B \se L(v)\big) &= \sum_{\substack{A \cap  B  = \emptyset}}\PP(A \se L(v))  +\sum_{\substack{A \cap  B  \neq \emptyset}}\PP\big(A \setminus B \se L(v)\big) \nonumber\\
    & \le \EE X +\sum_{\substack{A \cap  B  \neq \emptyset}}\PP\big(A \setminus B \se L(v)\big)\label{eq:prob_sum_split}
\end{align}
We now find an upper bound for the expression on the right hand side of~\eqref{eq:prob_sum_split}. To this end, for each $A \in \cB'$ with non-empty intersection with $B$ we define a variable $\ell = \ell(B, A)$ to represent the maximal index $j \in [N]$ such that $V( B \cap  A) \cap W_j \neq \emptyset$. Recall that each block in $\cB'$ has only elements $b$ with $p_{v \cup b} < 1$, so $\bo_1' = 0$ and $\ell \ge 2$. Therefore,
\begin{align*}\label{eq:disjoint_sum}
    \sum_{A \cap B \neq \emptyset} \PP\big(A \setminus B \se L(v)\big) 
    &  
    = \sum_{\ell = 2}^N \Bigg(\sum_{A : \ell(B, A) = \ell} \space\PP\big(A \setminus B \se L(v)\big)\Bigg).
    \end{align*}
    
By definition of $\bo'$, we have that $V(B)$ contains at most $\bo'_{\ell}$ vertices in $W_\ell$. By (WB \ref{WB2}), each of these may be contained in $V(A)$ for at most $\Delta_{\ell-1}$ sets $A \in \cB$. This gives at most $\bo'_{\ell}\Delta_{\ell-1}$ sets $A$ such that $V (B \cap A) \cap W_\ell \neq \emptyset$. For each $\ell$, we let $A_\ell \in \cB$ be the block which maximises the probability $\PP\big(A \setminus B \se L(v)\big)$ over all sets $A \in B$ with $\ell(B, A) = \ell$. We conclude that
\begin{align*}
    \sum_{A \cap B \neq \emptyset} \PP\big(A \setminus B \se L(v)\big) =  \sum_{\ell = 2}^N \Bigg(\sum_{A : \ell(B, A) = \ell} \space\PP\big(A \setminus B \se L(v)\big)\Bigg)
    &  \le \sum_{\ell = 2}^N\bo'_{\ell}\Delta_{\ell-1}  \PP(A_\ell \setminus B \se L(v)).
\end{align*}

We require our expression to be independent of our choice of $B$, so we let $A_*, B_* \in \cB'$ be the blocks which maximise $\Delta_{\ell(B, A)-1} \PP\big(A \setminus B \se L(v)\big)$ over all sets $A, B \in \cB'$ and let $\ell_* = \ell(B_*, A_*)$. We also note $\sum_{\ell = 2}^N\bo'_\ell \le (r-1)D$.  We therefore have

    \begin{align*}
    \sum_{A \cap B \neq \emptyset} \PP\big(A \setminus B \se L(v)\big) 
    & \le \sum_{\ell = 2}^N\bo'_{\ell}\Delta_{\ell_*-1}  \PP\big(A_* \setminus B_* \se L(v)\big)
    \le (r-1)D\Delta_{\ell_*-1}  \PP\big(A_* \setminus B_* \se L(v)\big).
    \end{align*}
By~\eqref{eq:prob_sum_split} we have
\[\sum_{A\in \cB'}\PP\big(A \setminus B \se L(v)\big) \le \EE X + (r-1)D\Delta_{\ell_*-1}  \PP\big(A_* \setminus B_* \se L(v)\big).\]
We now substitute this into~\eqref{eq:EX2} to obtain 
\begin{align*}
    \EE X^2 & \le \sum_{B\in \cB'} \bigg( \PP(B \se L(v))\cdot \big(\EE X + (r-1)D\Delta_{\ell_*-1}  \PP\big(A_* \setminus B_* \se L(v)\big)\big)\bigg)\nonumber\\
    & = \EE X \cdot \big(\EE X + (r-1)D\Delta_{\ell_*-1}  \PP\big(A_* \setminus B_* \se L(v)\big)\big)\nonumber\\
    & \le \EE X \cdot \max\big\{2\EE X, \space 2 (r-1)D\Delta_{\ell_*-1}  \PP\big(A_* \setminus B_* \se L(v)\big)\big\}.\label{eq:EX^2_val}
\end{align*}
If $2\EE X \ge 2 (r-1)D\Delta_{\ell_*-1}  \PP\big(A_* \setminus B_* \se L(v)\big)$, then $\PP(\cE') \ge 1/2 $ by~\eqref{eq:Chebyshev} and we are done. So we assume this is not the case. By~\eqref{eq:Chebyshev} we have 
\begin{equation}\label{eq:lb_PE}
    \PP(\cE') \ge \frac{\EE X}{2 (r-1)D\Delta_{\ell_*-1}  \PP\big(A_* \setminus B_* \se L(v)\big)}\space.
\end{equation}
As every $B \in \cB'$ has the same intersection pattern, we can write 
\begin{equation*}\label{eq:EX}
        \EE X = \sum_{B \in \cB'} \EE X_B = t'\cdot \PP\big(A_* \se L(v)\big) = t'\prod_{b \in A_*}p_{v \cup b}.
\end{equation*}
We substitute this into~\eqref{eq:lb_PE} and write $\PP\big(A_* \setminus B_* \se L(v)\big) = \prod_{b \in A_* \setminus B_*}p_{v \cup b}$ to give
\begin{align*}
    \PP(\cE') \ge \frac{ t'\prod_{b \in A_*} p_{v \cup b}} {2(r-1)D \Delta_{\ell_*-1} \prod_{b \in A_* \setminus B_*}p_{b \cup v}}
    & = \frac{t'\prod_{b \in (A_* \cap B_*)} p_{v \cup b}} {2(r-1)D \Delta_{\ell_*-1}}.
\end{align*}
We now let $b_* \in  (A_* \cap B_*)$ minimise $p_{v \cup b_*}$ and write
    $\PP(\cE') \ge t'\frac{ p_{v \cup b_*}^D} {2(r-1)D \Delta_{\ell_*-1}} $.
Note that $\pi_{i}(b_*) \neq 0$ for some $i \le \ell_*$ by the definition of $\ell_*$, and hence $p_{v\cup b} = p^* \prod_{i=1}^N\Delta_i^{{\pi_i(v \cup b_*)}} \ge p^* \Delta_k\Delta_{\ell_*}$. We have
\begin{align*}
    \PP(\cE') &\ge\ t'\frac{(p^* \Delta_k\Delta_{\ell_*})^D} {2(r-1)D \Delta_{\ell_*-1}} \\
    & = t'\frac{(p^* \Delta_k)^D} {2(r-1)D }\\
    & \ge  2^{{(r-1)}D}t'n^{D^{1-k}-1}(\log n)^2(\log \log n)^{rD},
\end{align*}
as required. 
\end{proof}

We show that this property implies Lemma \ref{lem:pseudorandom}.

\begin{proof}[Proof of Lemma \ref{lem:pseudorandom}]
    We calculate the probability that~\eqref{eq:lem} does not hold for a fixed $\cB, k$ and $j$ via an application of the Chernoff bound, Theorem~\ref{chernoff}, and Lemma \ref{lem:cand_for_some}, followed by a union bound over all choices of $\cB, k$ and $j$.

    If $k =1$, then for every $u \in W_{k, j}$ and $B \in \cB$ we have $\PP(B \se L(u)) = 1$, hence 
    \begin{equation*}
        |\{u \in W_{k, j} : \exists B \in \cB \text{ such that } B \se L(u)\}| = |W_{k, j}|
    \end{equation*} 
    and we are done. 

    Let $k \ge 2$ and fix some $j \in [\log n]$ and $t \in [n]$ and well-behaved set $\cB$ with $|\cB| = t$. We now argue that we may assume 
    \begin{equation}\label{eq:t_cond}
        t (\log n)^2(\log \log n)^D n^{D^{1-k}-1} < 1/2.
    \end{equation} 
    Suppose we prove the lemma under this condition on $t$. If we have some $t \in [n]$ such that~\eqref{eq:t_cond} does not hold, then we may choose an integer $t' < t$ such that $1/4 < t' (\log n)^2(\log \log n)^D n^{D^{1-k}-1} < 1/2$, and arbitrarily choose a subset $\cB' \subset \cB$ with $|\cB'| = t'$. Then by~\eqref{eq:lem} applied to $\cB'$,
    
    $$|\{u \in W_{k, j} : \exists B \in \cB' \subset \cB \text{ such that } B \se L(u)\}| \ge \frac{1} {16} |W_{k, j}|$$ 
    and we are done. We proceed with the assumption that $t$ satisfies~\eqref{eq:t_cond}.

    Let $U = W_{k, j} \setminus V(\cB)$ be the set of vertices of $W_{k, j}$ disjoint from $\cB$. For each $u \in U$, let $\cE_u$ be the the event that there exists some $B \in \cB$ such that $B \se L(u)$. By Lemma \ref{lem:cand_for_some}, we can write
    \begin{align*}
    \EE|\{u \in U:  \exists B \in \cB \text{ such that } B \se L(u)\}| &= \sum_{u \in U} \PP(\cE_u) \\
    & \geByRef{eq:t_cond}  t (\log n)^2(\log \log n)^D n^{D^{1-k}-1}|U| \\
        & \ge \frac t{\log n} (\log n)^2(\log \log n)^D n^{D^{1-k}-1} 3^Drn^{1-D^{-k}} \\
        &> 
        t(\log n)(\log \log n)^Dn^{D^{1-k}-D^{-k}},
    \end{align*}
    where the second inequality uses that $|U| \ge \frac 1 2 |W_{k, j}|$ by (WB\ref{WB3}). 
    Note that
    \begin{equation*}
        \frac{t (\log  n)^{2}(\log\log n)^{D}n^{D^{1-k}-1}/4}{t(\log n)(\log \log n)^Dn^{D^{1-k}-D^{-k}}} = \frac{ \log n}{4n^{1-D^{-k}}},
    \end{equation*}
    which tends to zero. Hence, if the inequality in~\eqref{eq:lem} does not hold, then the number of vertices in $U$ such that there is some $B \in \cB$ such that $B \se L(u)$ is less than half its expectation, which by the Chernoff bound in Theorem \ref{chernoff} happens with probability at most 
    \begin{align*}
        2 \exp\left(- \frac 1 {16}t(\log n)(\log \log n)^Dn^{D^{1-k}-D^{-k}}\right) 
         < 2n^{-\frac 1{16}t(\log\log n)^D}.
    \end{align*}

    For our fixed $t$, we now take a union bound over all choices of $k, j$ and $\cB$ with $|\cB| = t$. Each $B\in \cB$ can be determined by a vector of length $D$, whose entries are either from $\binom{V(\Gamma(n, r, D))}{r-1}$ or an empty symbol (if it contains less than $D$ sets). Arranging these vectors in a matrix with $t$ columns defines some $\cB$. As $|V(\Gamma(n, r, D)| \le 2 \cdot 3^Drn$, there are at most 
    \begin{equation*}
        {\binom{2\cdot3^Drn + 1}  {r-1}}^{tD} \le \left(\frac{e(2\cdot3^Drn + 1)}{r-1} \right)^{rtD} <n^{2rtD},
    \end{equation*}
    choices for $\cB$, where the first inequality is by~\eqref{eq:stirling} and the last inequality holds for large $n$. There are at most $N \le  \log \log n$ choices of $k$ and at most $\log n$ choices for $j$, so the probability that~\eqref{eq:lem} does not hold is at most 
    \begin{equation*}
        (\log \log n ) (\log n) \cdot 2n^{2rtD-\frac 1{16}t(\log\log n)^D} < n^{-2}.
    \end{equation*}
    Finally, taking a union bound over all $n$ choices of $t$ gives that the probability~\eqref{eq:lem} does not hold is $o(1)$, as required.
\end{proof}

\section{Concluding Remarks}

We prove bounds on the number of edges required in a hypergraph that is universal for the class $\cH^{(r)}(n,D)$ of $D$-degenerate graphs on~$n$ vertices for $D>1$ that are tight up to a logarithmic factor.
It would be interesting to determine if it is possible to remove this logarithmic factor. Similarly, it would be interesting to obtain analogous results for spanning universality.

We note that our upper-bound result (Theorem~\ref{thm:main}) does not apply to $1$-degenerate $r$-graphs. However, we believe with some minor adjustments our construction one could obtain an analogous bound for $D = 1$. It is not clear, if the additional logarithmic term in such a construction is needed in the case that $r>2$ as
the upper bound constructions given for the graph case in \cite{Chung_Graham_1978, kim2025} are not immediately generalisable to hypergraphs.

Another interesting direction would be to consider universality results for hypergraph of bounded density, generalising the results of~\cite{alon2024universalitygraphsboundeddensity}.

\medskip

\bibliographystyle{plainurl}
\bibliography{Citations}

\end{document}